\documentclass[12pt]{amsart}
\usepackage{amsmath,amssymb,txfonts}
\usepackage{amssymb}
\usepackage{amsxtra}
\usepackage{amsmath}
\usepackage{txfonts}
\def\XXint#1#2#3{{\setbox0=\hbox{$#1{#2#3}{\int}$ }
\vcenter{\hbox{$#2#3$ }}\kern-.6\wd0}}

\newcommand{\R}{\mathbb{R}} 
 \newcommand{\C}{\mathbb{C}}

\theoremstyle{plain} \newtheorem{theorem}{Theorem}[section]
\newtheorem{lemma}[theorem]{Lemma}

\theoremstyle{definition} \newtheorem{definition}[theorem]{Definition}
\theoremstyle{remark} \newtheorem{remark}{Remark}

\begin{document}

\title[Uniqueness of nonnegative weak solution to $u^p\le(-\Delta)^\frac{\alpha}{2}u$ on $\mathbb R^N$]
{Uniqueness of nonnegative weak solution to $u^p\le(-\Delta)^\frac{\alpha}{2}u$ on $\mathbb R^N$}

\author{Yuzhao Wang}
\address{Department of Mathematics and Physics, North China Electric Power University, Beijing 102206, China}
\curraddr{Department of Mathematics and Statistics, Memorial
University, St. John's, NL A1C 5S7, Canada}
\email{wangyuzhao2008@gmail.com}
\author{Jie Xiao}
\address{Department of Mathematics and Statistics, Memorial University, St. John's, NL A1C 5S7, Canada}
\email{jxiao@mun.ca}

\thanks{YW was supported by NSFC No. 1120143 and AARMS Postdoctoral Fellowship; JX was supported by NSERC of Canada (FOAPAL \# 202979463102000) and URP of Memorial University
(FOAPAL \# 208227463102000)}


\begin{abstract} 
This note shows that under $(p,\alpha, N)\in (1,\infty)\times(0,2)\times\mathbb Z_+$ the fractional order differential inequality
$$
(\dagger)\quad u^p \le (-\Delta)^{\frac{\alpha}{2}} u\quad\hbox{in}\quad\mathbb R^{N}
$$
has the property that if $N\le\alpha$ then a nonnegative solution to $(\dagger)$ is unique, and if $N>\alpha$ then the uniqueness of a nonnegative weak solution to $(\dagger)$ occurs when and only when $p\le N/(N-\alpha)$, thereby innovatively generalizing Gidas-Spruck's result for $u^p+\Delta u\le 0$ in $\R^N$ discovered in \cite{GS}. 
\end{abstract}
\maketitle
\tableofcontents

\section{Statement of Theorem}\label{s0}
\setcounter{equation}{0}

In their 1981 paper \cite{GS}, B. Gidas and J. Spruck discovers that the Laplacian differential inequality on Euclidean $N$-space $\mathbb R^N$ with $N\in\mathbb Z_+=\{1,2,3,...\}$:
$$
(\ddagger)\quad u^{p}\le -\Delta u\quad\hbox{under}\quad p\in (1,\infty)
$$
obeys the following law: if $N\le 2$ then a nonnegative weak solution to $(\ddagger)$ is unique; and if $N>2$ then the uniqueness of a nonnegative weak solution to $(\ddagger)$ happens when and only when $p\le N/(N-2)$. 

The fact that this result is extendable to more general differential inequalities in $\mathbb R^N$ and even Riemannian manifolds has been observed now by several authors with a variety of arguments; see \cite{CDP, CM, MP1, MP2, MP3, P1, P2, SZ, GS, S1, S2, S3, WX}.  

In this note, we will establish an analogue of Gidas-Spruck's discovery for fractional order Laplacian inequality. To state our main result, let us recall some notations and facts on fractional Laplacian.

For $0<\alpha\le 2$ let $(-\Delta)^{\frac{\alpha}{2}}$ be the $\frac{\alpha}{2}$ (fractional order) Laplacian which is determined via the Fourier transformation $\hat{(\cdot)}$: if $f$ satisfies
\begin{equation}
\label{int-con}
\int_{\R^N}\frac{|f(x)|}{(1+|x|)^{N+\alpha}}\, dx < \infty
\end{equation}
then
\begin{equation}
\label{fra-diff-fo}
\widehat{(-\Delta)^{\frac{\alpha}{2}} f} (\xi) = |\xi|^\alpha \hat f(\xi)=|\xi|^\alpha\int_{\mathbb R^n} e^{-i\xi \cdot x} f(x)\, dx\quad\forall\quad \xi\in\mathbb R^N;
\end{equation}
alternatively, one has 
\begin{equation}
\label{fra-diff}
(-\Delta)^{\frac{\alpha}{2}} f(x) = C_{N,\alpha}\ \hbox{p.v.} \int_{\R^N}\frac{f(x)-f(y)}{|x-y|^{N+\alpha}}\, dy\quad\forall\quad x\in\mathbb R^N
\end{equation}
with $ C_{N,\alpha}$ being a normalization constant. The fractional Laplace operators arise in stochastic theory as the operators associated with symmetric $\alpha$-stable Levy processes. Applications also can be found in financial mathematics, such as the pricing model for American options, we refer the reader to \cite{BL,CT} and reference therein.

\begin{theorem}
\label{t1}
For $(p,\alpha,N)\in (1,\infty)\times(0,2)\times\mathbb Z_+$, the fractional differential inequality
\begin{equation}\label{eq:2} 
u^p\le(-\Delta)^\frac{ \alpha}{2} u\quad \hbox{in}\quad\mathbb R^N,
\end{equation}
 enjoys the following property:

{\rm (i)} If $N\le\alpha$, then a nonnegative weak solution to \eqref{eq:2} is unique;

{\rm (ii)} If $N>\alpha$, then the uniqueness of a nonnegative weak solution of \eqref{eq:2} occurs when and only when $p\le N/(N-\alpha)$.
\end{theorem}

\begin{remark}
The notion of a nonnegative weak solution to \eqref{eq:2} is deferred to Definition \ref{d2}.
\end{remark}

Notice that the case $\alpha=2$ of Theorem \ref{t1} agrees with Gidas-Spruck's uniqueness result stated above. To prove Theorem \ref{t1}, we utilize Section 2 to introduce three new results of independent interest: the first is Lemma \ref{int-iq} which, as a local version of (\ref{eq:2}), is established by Caffarelli-Silvestre's $\alpha$-extension in \cite{CS}; the second is Lemma 
\ref{es-ext} that lifts $L^p(\mathbb R^n)$ to a weighted function space in the upper-half-space via the $\alpha$-extension; and the third is Lemma \ref{int-iqq} which exists as a mixed trace inequality (inspired more or less by \cite{X}) for the $\alpha$-extension of a nonnegative weak solution to (\ref{eq:2}). With the help of those three lemmas, Section 3 is designed to validate Theorem \ref{t1} - in particular - Theorem \ref{t1}(i) and the ``when" part of Theorem \ref{t1}(ii) are demonstrated by showing ``$u\ge 0\Rightarrow u=0$", and the ``only when" part of Theorem \ref{t1}(ii) is checked through constructing a positive solution based on the fundamental solution to $(-\Delta)^{\frac{\alpha}{2}} u = 0$ in $\mathbb R^N$.

\smallskip

\noindent{\it Notation.} In what follows, $\mathsf{U}\lesssim\mathsf{V}$ means that there is a constant $C > 0$ such that $\mathsf{U}\le C\mathsf{V}$. Moreover, $\mathsf{U} \approx \mathsf{V}$ stands for $\mathsf{U}\lesssim\mathsf{V}\lesssim\mathsf{U}$.

\section{Lemmas and Their Proofs}\label{s2}

\subsection{A reduction for $u^p\le(-\Delta)^\frac{\alpha}{2}u$}\label{s21} Such a reduction comes from another look at Caffarelli-Silvestre's extension technique (cf. \cite{CS}) that reduces the fractional Laplacian to a local problem through bringing one more variable into play. To be more precise, the fractional Laplacian can be characterized as a Dirichlet-Neumann operator for an appropriate differential equation of divergence form: if $\alpha\in (0,2)$, $u=u(x)$ is a function defined in $\R^N$, and $\omega=\omega(x,y)$ is a solution to the boundary value problem
\begin{equation}
\label{fra-ext}
\begin{cases}
-div(y^{1-\alpha} \nabla \omega)  =0\quad\text{in}\quad \R^{N+1}_+=\R^N\times(0,\infty);\\
\omega(x,0) = u(x)\quad\forall\quad x\in\R^N,
\end{cases}
\end{equation}
then there is a constant $C_\alpha$ depending on $\alpha$ such that
\begin{equation}
\label{fra-diff-loc}
\lim_{y\rightarrow 0^+} y^{1-\alpha} \frac{\partial \omega}{\partial y}(x,y) = - C_\alpha (-\Delta)^{\frac{\alpha}{2}} u(x)\quad\forall\quad x\in\R^N.
\end{equation}
Observe the divergence form in \eqref{fra-ext} is not only the Euler-Lagrange equation for the functional
\[
J(u) = \iint_{\R^{N+1}_+} |\nabla\omega|^2 y^{1-\alpha}\, dx dy
\]
with $\nabla$ being the gradient in $(x,y)$, but also it can be rewritten as
\begin{equation}
\label{fra-har}
\Delta_x \omega + {(1-\alpha)}{y}^{-1} \omega_y + \omega_{yy} = 0,
\end{equation}
which may be viewed as a harmonic function in $N+2-\alpha$ dimensions. This observation leads us to the fundamental solution to \eqref{fra-har}:
\begin{equation*}
\label{funda}
\Gamma(x,y) =\frac{\pi^{(N+2-\alpha)/2}\Gamma(\big(N-\alpha)/4\big)}
{\big(|x|^2 + |y|^2\big)^{\frac{N-\alpha}{2}}}\quad\hbox{under}\quad N-\alpha>1,
\end{equation*}
where $\Gamma(\cdot)$ is the standard gamma function. Furthermore, upon setting 
\[
P(x,y) = -y^{-1+\alpha} \partial_y \Gamma(x,y) = \left(\frac{\Gamma(\big(N+\alpha)/2\big)}{\pi^{N/2}\Gamma(\alpha/2)}\right) \frac{y^\alpha}{\big(|x|^2 + |y|^2\big)^{\frac{N+\alpha}{2}}} 
\]
be the Poisson kernel, we find that 
\begin{equation}
\label{exten}
\omega(x,y) = \int_{\R^N} P(x-z, y)u(z)\, dz,
\end{equation}
solves \eqref{fra-ext}. Due to \eqref{exten}, the solution $\omega(x,y)$ of \eqref{fra-ext} is also called the {$\alpha$-extension} of $u(x)$. 

One major advantage with the use of representing the fractional Laplacian as a divergence-form differential equation (\ref{fra-ext}-\ref{fra-diff-loc}) instead of a nonlocal operator \eqref{fra-diff-fo} or \eqref{fra-diff} is that one can localize the related problems. This method  has a great deal of applications in equations involving the fractional Laplacian; see \cite{CSS,CT,BCPS} and their references. Our current concern is the inequality \eqref{eq:2} (but not its equality) that is also nicely connected to the extension method in \cite{CS}. After extending the inequality to $\R^{N+1}_+$, it suffices to deal with a local equation \eqref{fra-har} by using an inequality as the boundary condition. In view of \eqref{fra-ext} and \eqref{fra-diff-loc}, the inequality \eqref{eq:2} can be rewritten as
\begin{equation}\label{in:loc}
\begin{cases}
div(y^{1-\alpha} \nabla \omega) = 0\quad\forall\quad (x,y)\in\R^{N+1}_+;\\
\displaystyle \lim_{y\rightarrow 0^+} y^{1-\alpha} \frac{\partial \omega}{\partial y}(x,y) + u^p(x) \le 0\quad\forall\quad x\in\R^N.
\end{cases}
\end{equation}
Naturally, $\omega$ can be extended to $\R^{N+1}$ via putting
\[
\tilde \omega(x,y) = 
\begin{cases}
\omega(x,y), \quad &\ \ \forall\ \ x\in\R^N\ \ \&\ \ y\ge 0;\\
\omega(x,-y), \quad & \ \ \forall\ \ x\in\R^N\ \ \&\ \ y< 0.
\end{cases}
\]

\begin{lemma}\label{int-iq} Given $(p,\alpha)\in (1,\infty)\times(0,2)$. Suppose that $u$ is a nonnegative  solution to \eqref{eq:2}, $\omega$ is its $\alpha$-extension, and $\nabla$ stands for the full gradient operator in $(x,y)\in\R^{N+1}$. If $|y|^{1-\alpha}|\nabla \tilde \omega(x,y)|\in L^2(\R^{N+1})$, then for any nonnegative continuous function $h$  satisfying $|y|^{1-\alpha} |\nabla h(x,y)|\in L^2(\R^{N+1})$, one has
\begin{equation}\label{int-iq0}
2\int_{\R^N} u^p(x) h(x,0)\,dx \le \iint_{\R^{N+1}} \nabla \tilde \omega(x,y) \cdot \nabla h(x,y) |y|^{1-\alpha}\, dx dy.
\end{equation}
\end{lemma}

\begin{proof}
Without losing of generality, we may assume that $h$ supports in the origin-centered Euclidean ball $B_R := B(0,R)$ in $\R^{N+1}$ with radius $R>0$. For any $\varepsilon > 0$, we utilize  $div(|y|^{1-\alpha} \nabla \omega) = 0$ and the Stokes formula to obtain
\begin{align*}
\iint_{B_R} \nabla & \tilde \omega(x,y) \cdot \nabla h(x,y) |y|^{1-\alpha} dx dy\\
 &=  \iint_{B_R\backslash \{ |y| < \varepsilon\}} + \iint_{B_R\cap \{ |y| < \varepsilon\}} \nabla \tilde \omega(x,y) \cdot \nabla h(x,y) |y|^{1-\alpha} dx dy \\
&= \iint_{B_R\backslash \{ |y| < \varepsilon\}} div(|y|^{1-\alpha} h \nabla  \tilde  \omega) dx dy + \iint_{B_R\cap \{ |y| < \varepsilon\}} \nabla \tilde \omega(x,y) \cdot \nabla h(x,y) |y|^{1-\alpha} dx dy\\
 &= - \iint_{B_R\cap \{ |y| = \varepsilon\}} \varepsilon^{1-\alpha} h(x,\varepsilon)  \frac{\partial  \tilde  \omega}{\partial y}(x,\varepsilon) dx + \iint_{B_R\cap \{ |y| < \varepsilon\}} \nabla \tilde \omega(x,y) \cdot \nabla h(x,y) |y|^{1-\alpha} dx dy.
\end{align*}
Upon noticing that the last integral goes to $0$ as $\varepsilon \rightarrow 0^+$ because of $|y|^{1-\alpha} |\nabla \omega(x,y)| \in L_{loc}^2(\R^{N+1})$ and that the inequality in \eqref{in:loc} gives
\[
 u^p(x)\le-\lim_{\varepsilon\rightarrow 0^+} \varepsilon^{1-\alpha} \frac{\partial \omega}{\partial y}(x,\varepsilon),
\]
we get (\ref{int-iq0}).
\end{proof}

\begin{remark}\label{r2} As a matter of fact, when $h(x,y)$ is the $\alpha$-extension of $h(x,0)$, the inequality \eqref{int-iq0} is equivalent to
\[
\int_{\R^N} u^p(x) h(x,0) dx \le \int_{\R^N} |\xi|^{\alpha} \hat u (\xi) \overline {\hat h} (\xi,0)  d\xi,
\]
which implies $u \in \dot H^{\alpha/2} \cap L^{p+1}$ immediately.
One way to verify this is to show that the corresponding energy functionals coincide through the energy functional formula in \cite{CS}:
\begin{equation}\label{energy}
\iint_{\R^{N+1}_+} |\nabla\omega(x,y)|^2 y^{1-\alpha}\, dxdy = \int_{\R^N} |\xi|^{\alpha} |\hat u (\xi)|^2\, d\xi,
\end{equation}
where $\omega$ is the $\alpha$-extension of $u$. A similar argument gives 
\[
\iint_{\R^{N+1}_+}  \nabla  \omega(x,y) \cdot \nabla w(x,y)  y^{1-\alpha}\, dxdy = \int_{\R^N} |\xi|^{\alpha} \hat u (\xi) \overline {\hat v} (\xi)  d\xi,
\]
where $v \in \dot H^{\alpha/2}$ and $w$ is its $\alpha$-extension. 
\end{remark}

More importantly, Lemma \ref{int-iq} suggests us to adopt the local version of a nonnegative weak solution to \eqref{eq:2} via the $\alpha$-extension. 
\begin{definition}\label{d2}
We say that $u$ on $\R^N$ is a nonnegative weak solution of \eqref{eq:2} provided that $u$ is a nonnegative function and its extension $\omega$ satisfies both $|y|^{1-\alpha} |\nabla \omega(x,y)|\in L^2(\R^{N+1})$ and 
\begin{equation}\label{weak-sol}
\int_{\R^N} u^p h|_{y=0}\,dx \le \iint_{\R^{N+1}_+} \nabla  \omega \cdot \nabla h \ \ |y|^{1-\alpha}\, dx dy,
\end{equation}
for any compactly supported nonnegative function $h$ satisfying $|y|^{1-\alpha} |\nabla h(x,y)|\in L^2(\R^{N+1})$.
\end{definition}

\subsection{An extension for $L^p(\mathbb R^n)$}\label{s22} This is presented in the following result.

\begin{lemma}
\label{es-ext} For $(p,\alpha)\in(1,\infty)\times(0,2)$ and $u\in L^p(\R^N)$, let 
$\omega(x,y)$ be given by (\ref{exten}). Then
\begin{equation}
\label{es-ext1}
\left(\iint_{\R^{N+1}_+} |\omega(x,y)|^{\frac{(N + 2 -\alpha)p}{N}} y^{1-\alpha}\,dxdy \right)^{\frac{N}{(N + 2 -\alpha)p}}\lesssim\|u\|_{L^{p}(\R^N)}.
\end{equation}
\end{lemma}

\begin{proof} The argument is split into two parts.

\smallskip

{\it Part 1:\ $2N/(N+2-\alpha)\le p\le\infty$}. This follows from the well-known complex interpolation theorem. So it suffices to prove the endpoint cases, which are $(\infty, \infty)$-estimate and $(2, \frac{2N}{N+2-\alpha})$-estimate. 
	
Let us first consider the $(\infty,\infty)$-estimate. According to \eqref{exten}, one has 
\[
\omega(x,y)=\left(\frac{\Gamma(\big(N+\alpha)/2\big)}{\pi^{N/2}\Gamma(\alpha/2)}\right) 
\int_{\R^N}\frac{y^{\alpha}u(\xi)}{(|x-\xi|^2 + y^2)^{\frac{N+\alpha}2}}\,d\xi.
\]
It is easy to get
\[
\int_{\R^N} \frac{y^\alpha}{(|x|^2 + y^2)^{\frac{N+\alpha}2}}\, dx=\int_{\R^N} \frac{1}{(1 + |x|^2 )^{\frac{N+\alpha}2}}\, dx < \infty,
\]
and then by Young's inequality
\[
\|\omega(x,y)\|_{L^\infty_{x,y}(\R^{N+1})} \lesssim \|u(x)\|_{L^\infty(\R^N)}.
\]
This in turn implies
\begin{equation}
\label{Linfty}
\|\omega(x,y)\|_{L^\infty_{x,y}(\R^{N+1}_+,y^{1-\alpha}dxdy)} \lesssim \| u(x)\|_{L^\infty_x}.
\end{equation}

Next, we make the $L^p$-estimate with $p = \frac{2N}{N+2-\alpha}$. After taking the Fourier transform, \eqref{fra-har} becomes
\[
-|\xi|^2 \hat \omega (\xi,y) + ({1-\alpha})^{-1}y\hat\omega_y(\xi, y) + \hat\omega_{yy}(\xi,y) = 0,
\]
which is an ordinary differential equation for each $\xi$. Suppose now that $\phi: [0,\infty) \rightarrow \R $ is the minimizer of the functional
\[
J(\phi) : = \int_{0}^\infty (|\phi|^2+|\phi'|^2) y^{1-\alpha}\, dy\quad\hbox{subject\ to}\quad \phi(0)=1.
\]
Then $\phi$ solves the following equation (cf. \cite{CS}),
$$
\begin{cases}
- \hat \phi (y) + \frac{1-\alpha}{y} \hat \phi_y (y) + \phi_{yy} (y) =  0;\\
\phi(0) =  1;\\
\lim_{y\rightarrow\infty } \phi(y) =  0.
\end{cases}
$$
Note that
\[
\hat \omega (\xi,y) = \hat u (\xi) \phi(|\xi|y).
\]
Thus the energy functional of $\omega$ is
\begin{equation}\label{L2}
\begin{split}
\iint_{\R^{N+1}_+} |\omega(x,y)|^2 y^{1-\alpha}\, dxdy = & \int_{\R^N} \int_0^\infty |\hat \omega(\xi,y)|^2 y^{1-\alpha}\, dy d\xi\\
= & \int_{\R^N} \int_0^\infty |\hat u (\xi)|^2 |\phi(|\xi|y)|^2 y^{1-\alpha}\, dy d\xi \\
= & \int_{\R^N}  |\xi|^{\alpha -2} |\hat u (\xi)|^2 \int_0^\infty|\phi(y)|^2 y^{1-\alpha}\, dy d\xi\\
\le & J(\phi)  \|u\|^2_{\dot H^{\frac{\alpha}{2}-1}}\\ 
\le & J(\phi)  \|u\|^2_{L^\frac{2N}{N+2-\alpha}},
\end{split}
\end{equation}
where we have used the fractional Sobolev imbedding for $\frac{\alpha}{2}-1 < 0$ in the last step. 

The desired inequality follows from putting the above two estimates together and using \cite[Theorem 1.3.4]{Gra} with 
$$
(X,d\mu) = (\R^{N+1}, y^{1-\alpha}dxdy);\ \ (Y,dm) = (\R^N, dx);\ \ (q_0,p_0,q_1,p_1) = \left(2, \frac{2N}{N+2-\alpha} ,\infty, \infty\right)
$$
and $T: u\mapsto\omega$ being the extension operator.

\smallskip

{\it Part 2:\ $1<p<2N/(N+2-\alpha)$}. This follows from the well-known real interpolation theorem. So it suffices to check the weak-type estimate at the endpoint $p=1$.

Given a space $X$ and a measurable function $f: X\rightarrow \C$, we define the distribution function $\lambda_f: \R_+ \rightarrow [0,+\infty]$ by the formula
\[
\lambda_f(t) = \mu (\{x\in X: |f(x)| \ge t\}).
\]
Recall the weak $L^p$-norm $\|f\|_{L^{p,\infty}(X)}$ by
\[
\|f\|^p_{L^{p,\infty}(X,d\mu)} : = \|t^p\lambda_f(t)\|_{L^\infty(\R_+)}.
\]
We say that $T$ is of strong-type $(p,q)$ if 
\[
\|Tf\|_{L^q(X,d\mu)} \lesssim\|f\|_{L^p(Y,dm)},
\] 
and of weak-type $(p,q)$ if the above inequality holds whenever $L^q$ is replaced by $L^{q,\infty}$.

Suppose now $d\mu= y^{1-\alpha} dxdy$. Then the weak $L^{r,\infty}$-norm has an equivalent counterpart
\begin{equation}\label{weakp}
\|f\|_{L^{r,\infty}} \approx \sup_{E} \mu(E)^{-\frac{1}{r'}} \left|\int_E f \, d\mu \right|,
\end{equation}
where the supremum is taken over $\mu$-measurable set $E\subset\R^{N+1}_+$. If $r_0 = \frac{N + 2 -\alpha}{N}$, then for $Tu=\omega$ we have 
\begin{equation*}
\begin{split}
\|Tu\|_{L^{r_0,\infty}} \approx & \sup_{E} \mu(E)^{-\frac{1}{r_0'}} \left|\int_E Tu \, d\mu \right|.
\end{split}
\end{equation*}
For the right side of the last equivalence, we use Fubini's theorem and \eqref{weakp} to achieve
\begin{equation*}
\begin{split} \mu(E)^{-\frac{1}{r_0'}} \left|\int_E Tu \, d\mu \right| \le &  \mu(E)^{-\frac{1}{r_0'}} \int_E \int_{\R^N} P(x-z,y)|u(z)|\,dz \, d\mu \\
 = & \left(\sup_{\tilde z} \mu(E-\tilde{z} )^{-\frac{1}{r_0'}} \int_{E-\tilde{z}} P(x,y) \, d\mu \right)  \|u\|_{L^1} \\ 
\lesssim & \|P\|_{L^{r_0,\infty}} \,  \|u\|_{L^1},
\end{split}
\end{equation*}
where $\tilde{z} = (z,0) \in \R^{N+1}_+$.

Now for the weak $(r_0,1)$-estimate of $T$, it is sufficient to bound 
$$
\|P\|_{L^{r_0,\infty}(X,d\mu)}\quad\hbox{for}\quad P(x,y)=\left(\frac{\Gamma(\big(N+\alpha)/2\big)}{\pi^{N/2}\Gamma(\alpha/2)}\right)
\frac{y^{-N}}{(|x/y|^2+1)^\frac{N+\alpha}{2}},
$$ 
namely,
\begin{equation}
\label{P}
\|P\|^{r_0}_{L^{r_0,\infty}(X,d\mu)} =  \sup_{\lambda>0} \lambda^{r_0} \mu(\{(x,y)\in\R^{N+1}_+:\ P(x,y) > \lambda\}). 
\end{equation}
Noticing
$$
P(x,y) =\left(\frac{\Gamma(\big(N+\alpha)/2\big)}{\pi^{N/2}\Gamma(\alpha/2)}\right)\frac{y^{\alpha}}{\big(|x|^2 + y^2 \big)^{ \frac{N+\alpha}{2}}} \lesssim {|X|^{-N}}
\quad\hbox{where}\quad X = (x,y),
$$
we find that the right side of \eqref{P} can be controlled by 
\begin{equation*}
\begin{split} 
\mu(\{(x,y)\in\R^{N+1}_+:\ P(x,y) > \lambda\}) &\lesssim  \mu \left(\{X \in \R^{N+1}_+:\ |X|^{-N} \gtrsim \lambda\}\right)\\
&\approx \int_{|X| \lesssim \lambda^{-\frac1N}} \,d\mu \\
&\lesssim  \left(\int_{|x|\lesssim \lambda^{-\frac1N}} \,dx\right)\left(\int_{y \lesssim \lambda^{-\frac1N}} y^{1-\alpha}\,dy\right) \\
&\approx  \lambda^{r_0},
\end{split}
\end{equation*}
thereby getting 
\[
\|P\|_{L^{r_0,\infty}(X,d\mu)} \lesssim 1,
\]
which implies that $T$ is of the weak $(r_0,1)$-estimate. Therefore the strong-type estimate desired in {\it Part 2} follows from the strong $(2, \frac{2N}{N+2-\alpha})$-estimate proved in {\it Part 1} and Marcienkiewicz's real interpolation method \cite[Theorem 1.3.2]{Gra}.
\end{proof}

\begin{remark}
	In fact, if we interpolate between the weak $(r_0,1)$-estimate and the strong $(\infty,\infty)$-estimate, the strong-type estimate for $1<p\le \infty$ can be obtained. However, the reason that we split the argument into two parts is that not only the proof of the first part is more direct, but also the exploited energy structure is a very powerful tool and has its own interest.
\end{remark}

\subsection{A mixed trace estimation for $u^p\le(-\Delta)^\frac{\alpha}{2}u$}\label{s23} Such a type of estimate is given below.

\begin{lemma}\label{int-iqq} Let $(p,\alpha)\in(1,\infty)\times(0,2)$ and $\varphi$ be a smooth function in $\R^{N+1}_+$ with the compact support $\hbox{supp}(\varphi)$ such that $0\le\varphi\le 1$ and $\varphi=1$ in an nonempty open subset of $\hbox{supp}(\phi)$. If $u$ is a nonnegative weak solution to \eqref{eq:2} and $\omega$ is its $\alpha$-extension, then for any $0<t\ll 1$ and $s\gg 1$ there is a constant $C_{s,N}>0$ depending on $s$ and $N$ such that 
\begin{equation}
\label{int-iq1}
 \int_{\R^{N}} \varphi^s|_{y=0} \  u^{p-t} dx 
\le C_{s,N} \left(\iint_{\R^{N+1}_+} \omega^{p'-t'} \varphi^s\,y^{1-\alpha} dxdy \right)^{\frac{1-t}{p'-t'}}   \left( t^{-\frac{p'}{p'-1}}  \iint_{\R^{N+1}_+}  |\nabla \varphi |^{\frac{2p' - 2t'}{p' -1 + t - t'}}y^{1-\alpha} dxdy \right)^{\frac{p'-1 +t-t'}{p'-t'}},
\end{equation}
where 
\[
p' = \frac{(N + 2 -\alpha)p}{N} \quad \& \quad  t' = \frac{(N + 2 -\alpha)t}{N}.
\]
\end{lemma}

\begin{proof} Since $u$ is nonnegative, so is $\omega$. For constructing a proper test function, we introduce a small number
$$
0<\delta\ll 1,
$$
and let 
\begin{equation}
\label{test}
\omega_\delta = \omega + \delta\quad\&\quad \psi(x,y) = \varphi(x,y)^s \omega_\delta(x,y)^{-t},
\end{equation}
where $ 0<t\ll 1$, $ s\gg 1$. A simple calculation derives
\begin{equation}
\label{test1}
\nabla \psi = -t \omega_\delta^{-1-t} \varphi^s \nabla \omega + s \omega_\delta^{-t} \varphi^{s-1} \nabla \varphi.
\end{equation}
Using the compactness of $\hbox{supp}(\varphi)$, we obtain $\psi \in L^{p+1}(\R^{N+1})$. And since $\omega_\delta$ is uniformly away from $0$,  $\omega^{-1}_\delta$ is uniformly bounded from above, 
whence getting that for the fixed $t>0$ and $s>1$ 
\[
|\nabla \psi(x,y)|^2 y^{1-\alpha} \, dxdy \in L^{2}(\R^{N+1}_+).
\]

The foregoing analysis tells us that such a $\psi$ can be chosen a legal test function $h$ for (\ref{weak-sol}). So, by \eqref{weak-sol} and \eqref{test1} we obtain
\begin{equation}
\label{test-1}
t\iint_{\R^{N+1}_+} \omega_\delta^{-1-t} \varphi^s |\nabla \omega|^2\,y^{1-\alpha} dxdy + \int_{\R^{N}} \varphi^s|_{y=0} (u+\delta)^{p-t}\, dx \le s\iint_{\R^{N+1}_+} \omega_\delta^{-t} \varphi^{s-1} (\nabla \omega\cdot \nabla\varphi)\,y^{1-\alpha}dxdy.
\end{equation}
By the H\"older inequality and then Young inequality, the right hand of \eqref{test-1} can be estimated as follows:
\begin{equation*}
\begin{split}
s\iint_{\R^{N+1}_+} & \omega_\delta^{-t} \varphi^{s-1} \nabla \omega \cdot\nabla \varphi\, y^{1-\alpha}dxdy \\
& = \iint_{\R^{N+1}_+} (t \omega_\delta^{-1-t} \varphi^s)^{\frac{ 1}{2}}  (\nabla \omega\cdot \nabla \varphi)  s t^{-1/2} \omega^{\frac{1-t}{2}} \varphi^{\frac{s}2 -1}\, y^{1-\alpha}dxdy \\
& \le \left( t \iint_{\R^{N+1}_+}\omega_\delta^{-1-t} \varphi^s |\nabla \omega|^2\,y^{1-\alpha} dxdy\right)^{1/2}  \left( s^2 t^{-1} \iint_{\R^{N+1}_+}\omega_\delta^{1-t} \varphi^{s-2} |\nabla \varphi |^2\, y^{1-\alpha} dxdy\right)^{1/2}  \\
& \le \frac{ t}{2}\iint_{\R^{N+1}_+} \omega_\delta^{-1-t} \varphi^s |\nabla \omega|^2\,y^{1-\alpha} dxdy + \frac{s^2 t^{-1}}{2}\iint_{\R^{N+1}_+} \omega_\delta^{1-t} \varphi^{s-2} |\nabla \varphi |^2\, y^{1-\alpha} dxdy 
\end{split}
\end{equation*}
Bringing this last inequality into \eqref{test-1} we obtain
\begin{equation}
\label{test-1'}
\frac{t}2\iint_{\R^{N+1}_+} \omega_\delta^{-1-t} \varphi^s |\nabla \omega|^2\, y^{1-\alpha}dxdy + \int_{\R^{N}} \varphi^s|_{y=0} (u+\delta)^{p-t}\, dx \le \frac{s^2 t^{-1}}{2} \iint_{\R^{N+1}_+} \omega_\delta^{1-t} \varphi^{s-2} |\nabla \varphi |^2\, y^{1-\alpha} dxdy.
\end{equation}
By the H\"older inequality, the right side of \eqref{test-1'} can be estimated as follows:
\begin{equation}\label{test-12}
\begin{split}
2^{-1}s^2 & t^{-1} \iint_{\R^{N+1}_+} \omega_\delta^{1-t} \varphi^{s-2} |\nabla \varphi |^2\,y^{1-\alpha} dxdy \\
& =  \iint_{\R^{N+1}_+} \omega_\delta^{1-t} \varphi^{\frac{1-t}{p' -t'}s} \Big(2^{-1}s^2 t^{-1}  \varphi^{s- 2-  \frac{1-t}{ p' -t'}s } |\nabla \varphi |^2\Big)\, y^{1-\alpha} dxdy \\
& \le  \| \omega_\delta^{1-t} \varphi^{\frac{1-t}{p' -t'}s} \|_{L^{p_1}(\R^{N+1}_+, y^{1-\alpha}dxdy)} \|2^{-1}s^2 t^{-1}  \varphi^{s- 2- \frac{1-t}{ p' -t'}s } |\nabla \varphi |^2\|_{L^{p_1'}(\R^{N+1}_+, y^{1-\alpha}dxdy)},
\end{split}
\end{equation}
where 
\[
p_1 = 
\]
Note that for $p>1$, we have
\[
p' \ge \frac{N+2-\alpha}{N}\ \ \Longrightarrow (s^2 t^{-1} \big)^{\frac{p' -t'}{p' -1 + t - t'}} \le (s^2 t^{-1} \big)^{\frac{p' }{p' -1}}.
\]
And for $s\gg 1$ and $0< t \ll 1$, we have
\[
\varphi^{s-2 -\frac{1-t}{ p' -t'}s } \le 1.
\]
So, combining \eqref{test-1'} and \eqref{test-12} we finally obtain
\begin{align*}
\int_{\R^{N}} \varphi^s|_{y=0} &(u+\delta)^{p-t}\, dx \\
&\lesssim {\left(\iint_{\R^{N+1}_+} \omega_\delta^{p' -t'} \varphi^s\, y^{1-\alpha}dxdy \right)^{\frac{1-t}{p'-t'}}}{\left(\iint_{\R^{N+1}_+}  t^{-\frac{p'}{p'-1}}  \iint_{\R^{N+1}_+}  |\nabla \varphi |^{\frac{2p' - 2t'}{p' -1 + t - t'}}\, y^{1-\alpha}dxdy \right)^{\frac{p'-1 +t-t'}{p'-t'}}}.
\end{align*}
Letting $\delta \rightarrow 0$, we obtain \eqref{int-iq1}, thereby completing the proof.
\end{proof}

\section{Verification of Theorem}\label{s3}

\begin{proof}[Proof of the case $N\le\alpha$ or $N>\alpha\ \&\ p\le N/(N-\alpha)$] Assume that $u$ is a nonnegative weak solution to \eqref{eq:2}. Then $u$ satisfies \eqref{weak-sol}. If $\omega$ is the $\alpha$-extension of $u$, then Lemma \ref{int-iqq} holds. 

In what follows, we estimate the second factor of right side of \eqref{int-iq1} by selecting a series of appropriate test functions. To be more explicit, for a large number $R>0$ let us consider the function
\begin{equation}
\label{test0}
\varphi(X) = 
\begin{cases}
1, & |X| < R,\\
\left| \frac{X}{R} \right|^{-t}, & |X| \ge R,
\end{cases}
\end{equation}
where 
$$
X = (x,y) \in \R^{N+1}_+\ \ \&\ \ |X| = \sqrt{|x|^2 +y^2}.
$$
Since $R$ is big enough, $t =({\ln R})^{-1}$ is sufficiently small. For any $n\in\mathbb Z_+$ define a cutoff function $\eta_n$ by
\begin{equation}
\eta_n(X) =
\begin{cases}
1, & 0\le |X| \le nR;\\
2-\frac{|X|}{nR}, & nR \le |X| \le 2nR;\\
0, &|X| \ge 2nR,
\end{cases}
\end{equation}
and then consider the function 
\[
\varphi_n(X) = \varphi(X)\eta_n(X)
\]
so that $\varphi_n(X)$ approaches to $\varphi(X)$ from below as $n\rightarrow \infty$. Notice that
\[
\nabla \varphi_n = \varphi \nabla \eta_n + \eta_n \nabla \varphi.
\]
Thus, for any $a\ge 2$ one has 
\begin{equation}
\label{es1}
|\nabla \varphi_n |^a \lesssim |\varphi \nabla \eta_n|^a +| \eta_n \nabla \varphi|^a.
\end{equation}

Now we are ready to estimate the left side of \eqref{int-iq1} by choosing $\varphi_n$ to be the test functions. More precisely, we need to deal with the integral 
\begin{equation*}
\label{es2}
I_n(a) : = \iint_{\R^{N+1}_+} |\nabla \varphi_n|^a y^{1-\alpha}\, dxdy.
\end{equation*}
From \eqref{es1} we have
\begin{equation}
\label{es3}
\begin{split}
I_n(a) \lesssim & \iint_{\R^{N+1}_+} \eta_n^a|\nabla \varphi|^a y^{1-\alpha}\, dxdy + \iint_{\R^{N+1}_+} \varphi^a |\nabla \eta_n|^a y^{1-\alpha}\, dxdy \\
\lesssim & \iint_{\R^{N+1}_+\backslash B_R} |\nabla \varphi|^a y^{1-\alpha} dxdy +  \iint_{(B_{2nR}\backslash B_{nR})\cap\R^{N+1}_+} \varphi^a |\nabla \eta_n|^a y^{1-\alpha}\, dxdy.
\end{split}
\end{equation}
For the second integral in \eqref{es3}, we use $|\nabla \eta_n| \le ({nR})^{-1}$ to gain
\begin{equation}
\label{es4}
\begin{split}
\iint_{\R^{N+1}_+} \varphi^a |\nabla \eta_n|^a y^{1-\alpha}\, dxdy  
\lesssim & (nR)^{-a} \int_{(B_{2nR}\backslash B_{nR})\cap\R^{N+1}_+} \varphi^a |X|^{1-\alpha}\, dX \\
\lesssim & {(nR)^{-a}} \left(\sup_{(B_{2nR}\backslash B_{nR})\cap\R^{N+1}_+} \varphi^a \right) (nR)^{1-\alpha} |B(0,2nR)| \\
\lesssim & (nR)^{-a}\left(\frac{nR}{R}\right)^{-at} (nR)^{1-\alpha} (2nR)^{N+1} \\
\approx & n^{N+2-\alpha -a -at} R^{N+2-\alpha -a}.
\end{split}
\end{equation}
For the first integral in \eqref{es3}, we employ $|\nabla \varphi| \le R^t t |X|^{-1-t}$ to obtain that
\begin{equation}
\label{es5}
\begin{split}
\iint_{\R^{N+1}_+\backslash B_{R}} |\nabla \varphi|^a y^{1-\alpha}\, dxdy  
\lesssim & \int_{\R^{N+1}\backslash B_{R}} |\nabla \varphi|^a |X|^{1-\alpha}\, dX \\
\approx & \int_{R/2}^\infty R^{at} t^a r^{-a-at} r^{N+1-\alpha} dr \\
\approx & t^a R^{ -a +N+2-\alpha}.
\end{split}
\end{equation}
provided $a+at > N+2 -\alpha $. Combining \eqref{es4}, \eqref{es5} with \eqref{es3} yields
\[
I_n(a) \lesssim n^{N+2 -\alpha -a -at} R^{N+2-\alpha -a} +  t^a R^{ -a +N+2-\alpha},
\]
and consequently, 
\begin{equation}
\label{es6}
I_n(a) \lesssim n^{-at} +  t^a\quad\hbox{under}\quad a \ge N+2 -\alpha.
\end{equation}
Here, it should be pointed out that $I_n(a)$ is uniformly bounded in $R$ and $n$.

Now, from \eqref{int-iq1} and \eqref{es6} we obtain a constant $C>0$ depending only on $s$ and $a = \frac{p' -t'}{p'-1}$ such that 
\[
 \int_{\R^{N}} \varphi_n^s u^{p-t}\, dx 
\le C (n^{-at} +  t^a)\left(\iint_{\R^{N+1}_+} \omega^{p'-t'} \varphi_n^s\, y^{1-\alpha} dxdy\right)^{\frac{1-t}{p'-t'}}.
\]
Letting $n \rightarrow \infty$, we further obtain
\begin{equation}
\label{es06}
 \int_{\R^{N}} \varphi^s u^{p-t}\, dx 
\lesssim t^a\left(\iint_{\R_+^{N+1}} \omega^{p'-t'} \varphi^s\,y^{1-\alpha} dxdy \right)^{\frac{1-t}{p'-t'}}.
\end{equation}
Since $u\in L^{p+1}(\R^N)$, we conclude that $\omega \in L^{\frac{(p+1)(N+2-\alpha)}{N}}(\R^{N+1}_+, \ d\mu)$ by Lemma \ref{es-ext}, and that the integral of right side of \eqref{es06} is uniformly bounded in $t$. By letting $t\rightarrow 0+$, we discover
\[
\int_{\R^{N}} \varphi^s u^{p}\, dx  = \lim_{t\rightarrow 0+} \int_{\R^{N}} \varphi^s u^{p-t}\, dx \le \lim_{t\rightarrow 0+} t^a\left(\iint_{\R_+^{N+1}} \omega^{p'-t'} \varphi^s\,y^{1-\alpha} dxdy \right)^{\frac{1-t}{p'-t'}} =0,
\]
whence reaching $u=0$ provided $a  +at > N+2 -\alpha $.

Note that
\[
a = \frac{2(p' -t')}{p' -1}.
\] 
So, choosing $t$ to be sufficiently small one gets
\[
a+ at = \frac{2p'}{p'-1} + 2t' -\frac{2tt'}{p' -1} > \frac{2p'}{p'-1}.
\]
Thus we should guarantee 
\[
\frac{2p'}{p'-1} \ge N+2 -\alpha,\ \hbox{i.e.},\ 
(N-\alpha)p' \le N+2-\alpha.
\]
Recall that 
$$
p' = \frac{p(N+2-\alpha)}{N}.
$$ 
Thus we should have
\begin{equation*}
\label{non}
\frac{p(N+2-\alpha)(N-\alpha)}{N} \le N+2-\alpha,\ \hbox{i.e.},\ 
\frac{p(N-\alpha)}{N} \le 1,
\end{equation*}
as required.
\end{proof}

\begin{proof}[Proof of the case $p>N/(N-\alpha)$]\label{s32} Trivially, $u=0$ is a solution to \eqref{eq:2}. But, a positive solution of \eqref{eq:2} will be provided below. 

To do so, for $0<s<N$ let $I_s$ be the Riesz potential determined by
\[
I_s f = (-\Delta)^{-\frac{s}{2}}f.
\]
Actually, one has
\[
I_s f(x) = \left(\frac{\Gamma(N/2 - s/2)}{\pi^{N/2} 2^s \Gamma(s/2)}\right)\int_{\R^N} \frac{f(y)}{|x-y|^{N-s}}\, dy\quad\forall\quad x\in\R^N.
\]
Of course, the following fact is well-known:

\begin{equation}
\label{FT}
\widehat{|x|^{-N+s}}=c_{N,s}|x|^{-s}\quad\&\quad\widehat{I_sf}(x)=c_{N,s}\left(\frac{\Gamma(N/2 - s/2)}{\pi^{N/2} 2^s \Gamma(s/2)}\right)|x|^{-s}\hat f (x)
\end{equation}
with $c_{N,s}$ being a constant.

So, it is easy to find that the Riesz potential is the inverse of the fractional Laplacian in the sense of:
\[
I_\alpha ((-\Delta)^{\frac{\alpha}{2}} f) =(-\Delta)^{\frac{\alpha}{2}} ( I_\alpha f)  =f.
\]
By setting $f = \delta_0$ (the Dirac point mass at the origin) in the last formula we get
\[
(-\Delta)^{\frac{\alpha}{2}} ( I_\alpha \delta_0) = \delta_0,
\]
whence seeing that
\begin{equation}
\label{fundan}
W_\alpha(x) = I_\alpha\delta_0(x)=\left(\frac{\Gamma(N/2 - s/2)}{\pi^{N/2} 2^s \Gamma(s/2)}\right){|x|^{\alpha-N}}
\end{equation}
is the fundamental solution of $(-\Delta)^{\frac{\alpha}{2}}u=0$.

The main idea of the remaining argument is to perturb the fundamental solution \eqref{fundan} properly to construct a positive solution to \eqref{eq:2} when $p<N/(N-\alpha)$. However, the fundamental solution is not a proper solution due to its singularity at the origin. So, in order to handle this issue, we need a smooth cutoff function $\rho(x)$ that is defined by
\[
\rho(x) = 
\begin{cases}
1, & |x| \le 1;\\
\text{smooth and radially decreasing}, & 1\le |x| \le 2;\\
0, & |x| \ge 2.\\
\end{cases}
\]
It is easy to see that the function
\[
u(x) = \rho*W_\alpha(x) = \left(\frac{\Gamma(N/2 - s/2)}{\pi^{N/2} 2^s \Gamma(s/2)}\right)
\int_{\R^N} \frac{\rho(y)}{|x-y|^{N-\alpha}}\, dy\quad\forall\quad x\in\R^n
\]
is well-defined when $0<\alpha <N$, and solves the $\alpha$-fractional Laplacian equation, i.e.,
\[
(-\Delta)^{\frac{\alpha}{2}} u(x) = \rho(x)\quad\forall\quad x\in\R^N.
\]
Thus, for $0<\delta<N-\alpha$ we perturb $\rho*W_\alpha(x)$ at infinity to get
\[
u_\delta (x):= \rho*\frac1{|\cdot|^{n-\alpha -\delta}}(x) = \int_{\R^N} \frac{\rho(y)}{|x-y|^{N-\alpha-\delta}} dy\quad\forall\quad x\in\R^N.
\]
We next claim that $u_\delta$ solves the inequality \eqref{eq:2} under $p> \frac{N}{N-\alpha}$ when $|x|$ is large. We observe
\[
u_\delta (x) = \int_{|y| \le 2} \frac{\rho(y)}{|x-y|^{N-\alpha-\delta}} dy \approx \frac1{|x|^{N-\alpha-\delta}}
\quad\forall\quad |x| \gg 1,
\] 
whence getting 
\begin{equation}
\label{nonl}
u^p_\delta (x) \approx \frac1{|x|^{Np-\alpha p-\delta p}}\quad \forall\quad |x| \gg 1.
\end{equation}
On the other hand, we use \eqref{FT} to produce a constant $c>0$ such that
\[
\hat u_\delta (\xi) = c |\xi|^{-\alpha-\delta} \hat \rho (\xi),
\]
and thus 
\[
\widehat{(-\Delta)^{\frac{\alpha}{2}} u_\delta} (\xi) = c  |\xi|^{-\delta} \hat \rho (\xi),
\]
which in turn implies
\[
(-\Delta)^{\frac{\alpha}{2}} u_\delta (x) = c \rho * \frac1{|\cdot|^{N-\delta}}(x).
\]
Similarly, we have
\begin{equation}
\label{lin}
(-\Delta)^{\frac{\alpha}{2}} u_\delta (x) \approx \frac1{|x|^{N-\delta}}  \quad \forall\quad |x| \gg 1.
\end{equation}
If $p> \frac{N}{N-\alpha}$, then there exists a sufficiently small constant $\delta$ such that
$(N -\alpha -\delta) p > N -\delta$.
This inequality, along with \eqref{nonl} and \eqref{lin}, implies  
\[
u^p_\delta (x) \le (-\Delta)^{\frac{\alpha}{2}} u_\delta (x)  \quad \forall\quad |x| \gg 1,
\] 
thereby verifying the claim.

Since we have showed that there is a positive $R$ such that $u_\delta$ solves \eqref{eq:2} with $p > \frac{N}{N-\alpha}$ when $|x|\ge R$, for the rest part $|x| \le R$ we need to exploit the scaling structure of the inequality. To be more precise, observe that both $u_\delta $ and $(-\Delta)^{\frac{\alpha}{2}} u_\delta$ are continues and positive, and decay radially. So, letting 
\[
m = \max \left\{\frac{\displaystyle\max_{|x| \le R} u^p_\delta(x)}{\displaystyle\min_{|x| \le R} (-\Delta)^{\frac{\alpha}{2}} u_\delta(x)}, 1 \right\},
\]
we utilize the previous construction to deduce that $u_\delta$ solves the inequality
\begin{equation*}
\label{scal}
u^p  \le m (-\Delta)^{\frac{\alpha}{2}} u\quad\hbox{in}\quad\R^N.
\end{equation*}
Finally, choosing $u = m^{\frac1{1-p}} u_\delta$ we obtain a positive solution to \eqref{eq:2} on $\R^N$, as required.
\end{proof}

\end{document}